\documentclass[a4paper,12pt]{amsart}
\pdfoutput=1
\usepackage[protrusion=true,expansion=true]{microtype}
\usepackage[colorlinks=true,linkcolor=black,citecolor=black]{hyperref}
\usepackage{xcolor}
\usepackage[T1]{fontenc} 
\usepackage{mathtools}
\usepackage[all]{xy}
\usepackage[english]{babel}
\usepackage{frcursive} 
\usepackage[applemac]{inputenc}
\usepackage{eucal}
\usepackage{graphicx}
\usepackage{epsfig}
\usepackage{mathrsfs}
\usepackage{amssymb}
\usepackage{amsxtra}
\usepackage{enumerate}
\input cyracc.def 
 
\newtheorem*{theorem}{Theorem}

\newtheorem*{proposition}{ Proposition}
\newtheorem*{lemma}{ Lemma}

\newtheorem*{corollary}{Corollary}

\newtheorem*{definition}{Definition}

\theoremstyle{remark}

\def \1{\mathbb {1}}
\def \RM{\mathbb {R}}
\def \NM{\mathbb{N}}
\def \CM{\mathbb{C}}

\def \QM{\mathbb{Q}}

\def \Der {{\rm Der\,}}

\def \Aut {{\rm Aut\,}}

\def \p {{\rm exp\,}}

\def \d{\partial}
 
\def\a{\alpha}
\def\b{\beta}
\def\e{\varepsilon}

\def\l{\lambda}
\def\L{\Lambda}
\def\p{\varphi}

\def\G{\Gamma}   

\def \s{\sigma}

\def \to{\longrightarrow} 
\def \w{\wedge}
\def \alg{\mathfrak{g}}

\def \< {{\langle }}
\def \> {{\rangle }}
\def \( {\left( }
\def \) {\right) }

\newcommand{\Bt}{{\mathcal B}}

\newcommand{\Kt}{{\mathcal K}}
\newcommand{\Lt}{{\mathcal L}}
\newcommand{\Mt}{{\mathcal M}}

\newcommand{\Ot}{{\mathcal O}}

\newcommand{\Xt}{{\mathcal X}}

\newcommand{\Oth}{\widehat{\Ot}}
\newcommand{\Xth}{\widehat{\Xt}}
\renewcommand{\mod}{{\rm  mod\,}}

\title[Versal deformations of vector fields]{Versal deformations of\\ vector field singularities}
\author[M. Garay D. van Straten]{Mauricio GARAY and Duco van Straten}
\begin{document}
\begin{abstract}
When a singular point of a vector field passes through resonance, a formal 
invariant cone appears. In the seventies, Pyartli proved that for $(-1,1)$-resonance the cone 
is in fact analytic and is the degeneration of a family of invariant cylinders. In his thesis, 
Stolovitch established a new type of normal form and proved that for a simple resonance and under 
arithmetic conditions the cone is (the germ of) an analytic variety.  In this paper, we prove a 
versal deformation theorem for analytic vector fields with an isolated singularity over Cantor sets.
Our result implies that, under arithmetic conditions, the resonant cone is the degeneration of a 
set of invariant manifolds like in Pyartli's example. For the multi-Hopf bifurcation, that is for 
the $(-1,1)^d$-resonance, this implies the existence of vanishing tori carrying quasi-periodic 
motions generalising previous results of Chenciner and Li.
\end{abstract}
\maketitle
\tableofcontents
\section{Introduction}
In the formal neighbourhood of an analytic planar vector field one may find an invariant algebroid curve, 
classically called a {\em separatrix}. The study of these curves goes back to Briot-Bouquet, Dulac, Picard, 
Poincar\'e and was finally solved by  Camacho and Sad, who proved their existence in full 
generality~\cite{Briot_Bouquet,Camacho_Sad,Dulac_1912,Picard_tome_3,Poincare_these}.

These early investigations on separatrices were intimately connected to the concept of {\em normal forms and deformations} 
near equilibrium; the idea being that if a vector field can be reduced to a simple form by a change of variables, then it 
would be easy to find the separatrices. But due to the appearance of small denominators, the issue of convergence of the
required transformation turned out to be very non-trivial. In fact, the Poincar\'e-Siegel divergence theorems suggest 
that there is no analytic normal forms or versal deformations over a smooth base~\cite{Poincare_trois,Siegel_divergence}.

In the work of Kolmogorov, Siegel, Bruno and R\"ussmann profound results were obtained with regard to convergence questions 
outside the Poincar\'e domain~\cite{Bruno,Kolmogorov_KAM,Russmann,Siegel_vecteurs}.
As a rule, the normalising transformations diverge and only in cases of integrability one may expect convergence of the 
normalising transformations. The dichotomy theorems of Ilyashenko and Perez-Marco show that these integrable cases 
are exceptional~\cite{Ilyashenko_divergence,Perez_Marco_divergence,Perez_Marco}. For excellent overviews on these topics 
we refer the reader to~ \cite{Ilyashenko_Yakovenko,Martinet_Bourbaki,Stolovitch_progress}.

So in the light of KAM theory, the best one can hope for is the construction of a convergent versal 
deformations/normal forms {\em over a Cantor set of positive measure}. It  is the purpose of this paper to state and proves such a result. 

Unlike the case of Hamiltonian dynamics and KAM theory, the Cantor set we construct
is not invariant and integral curves may not be all contained inside the set. One may ask: what is the use a vector field over such 
a set?  The answer is quite simple: although the Cantor set is not invariant, it may 
very well contain an interesting subset of persistant invariant varieties. In this way we are led back to the very origins of normal forms alluded above.

Hyperbolic vector fields being linearisable in the $C^\infty$-category, divergence and small denominators are 
invisible in the smooth category~\cite{Sternberg_1958}~(incidentally this is not the case in the Gevrey category 
see~ \cite{Stolovitch_2013}). This underlines the relevance looking at vector fields in the complex domain,  as was 
traditionally done in the early works on the subject.

The simplest example is given by the $(1,-1)$ resonance, leading to the 
so-called {\em Hopf bifurcation}:
$$\left\{ \begin{matrix} \dot x=(-\a+\mu_1+\mu_2 xy) x \\
\dot y =(\a +\mu_1+\mu_2 xy)y
\end{matrix}\right.
$$
Classically one takes $\a$ is purely imaginary and restricts the flow to
the real subspace defined by $x =\bar y=z$. For $\mu=0$, the cone $xy=0$ is an analytic invariant variety that we 
call the {\em resonant cone}. In the real case, the curves $xy=\e$, that is the circles $|z|^2=\e$, are limit cycles 
for $\e>0$ stable under perturbations~\cite{Andronov_1929,Hopf_1942,Marsden_Hopf}. In the complex domain, these limit 
cycles vanishing at the singular point are real parts of the {\em invariant cylinders} $xy=\e$ that degenerate into the
resonant cone~\cite{Arnold_resonance,Arnold_versal}.  So we find here the classical situation of the {\em vanishing cycles} 
at an $A_1$ singularity~: a horizontal family of 1-cycles generating the homology of the cylinder, which vanishes at the 
singular point~(see for instance~\cite{AVGII,AVGL1}). Pyartli proved the stability  under perturbation of the analytic 
cylinders degenerating into a cone~\cite{Pyartli_resonance}.

It appears that Bruno thought Arnold had conjectured the existence of invariant families in higher dimensions for simple 
resonances and disproved this conjecture~\cite{Bruno_resonance}.  Incidentally, Bruno's result was misquoted by Arnold 
in his famous book on differential equation, where he attributes to Bruno and Pyartli the proof of the existence of such 
families in dimension 3!~(see \cite[Chapter 6, \S 36, E]{Arnold_edo}). What Arnold conjectured or at least hoped was the 
analyticity of the resonant cone for a single resonance and he wrote: {\em However, we can hope that these small perturbations do not destroy the invariant manifold $M_0$.}
 
Apparently unaware of Arnold's problem, \'Ecalle asserted that the conjecture holds, but only under an arithmetic condition 
on the eigenvalues of the vector field~\cite{Ecalle_92}.  Analyticity of the resonant cone was finally proved in 1994  by 
Stolovitch in his thesis. Stolovitch showed also the necessity of an arithmetic condition~\cite{Stolovitch_94}.  
Finally he extended the case of simple resonances to that of so-called {\em positive} resonances, but apparently
made a (minor) mistake in defining positivity~(see \ref{SS::assumptions}).

More recently Stolovitch proved a KAM theorem for vector fields, namely he showed the existence of a positive measure 
set of invariant manifolds, assuming a triviality condition of the Poincar\'e-Dulac normal form~\cite{Stolovitch_KAM}. 
For the case of the Hopf bifurcation, the triviality condition of Stolovitch means that the Poincar\'e-Dulac normal form 
defines a symplectic vector field, so rather than an isolated limit cycle we have a family of circles, the system is 
integrable and there is no bifurcation at all. Similarly for the multi-dimensional Hopf bifurcation, the condition of 
Stolovitch means that the system is Hamiltonian. Therefore in this case, the theorem implies the existence of a positive 
measure set of complex analytic invariant manifolds carrying a complex quasi-periodic motion. In the real elliptic case, 
these are the complexification of KAM tori vanishing at the singularity and the theorems of Stolovitch imply in particular 
that these form a set of positive measure around the singularity~(see also~\cite{Lagrange_KAM}).\\

From the standpoint of a singularist, the KAM theorem of Stolovitch means that there are strata in the versal deformation 
space on which invariant varieties form a positive measure set. So it is natural to ask if such a family could exist over 
the whole space. Such a result turns out to be a consequence of our versal deformation theorem. If we modify Arnold's 
conjecture  in the sense that the family of analytic varieties is not analytic but only continuous (and even Whitney $C^\infty$) over a set of positive measure, then it holds true. 

The invariant manifolds being toric manifolds, the result might be seen as a complex variant of the theory of quasi-periodic 
motions in the dissipative context at a singularity. The general theory of such quasi-periodic motions in the dissipative 
context has been established in the pioneering works of Broer, Bruno, Huitema, Sevryuk and Takens~(see~\cite{Broer_Huitema_Sevryuk_families,Broer_Huitema_Sevryuk_book,Bruno_dissipative}). Applying our result to the multiple Hopf bifurcation, we obtain 
the existence of quasi-periodic motions vanishing  at the singularity. This generalises a result obtained by  Li in dimension 
2 and 3~\cite{Li_Hopf}~(see also~\cite{Chenciner_Hopf} for the corresponding result for discrete dynamical systems).

In this paper we use, as Stolovitch,  positivity assumptions on the linear part of our vector field, but in fact we are 
confident that the results  can be proven under less restrictive conditions. It is for instance plausible that the methods 
for the classification on planar vector field due Malgrange, Martinet-Ramis, Moussu-Cerveau can be used~\cite{Malgrange_Bourbaki, Martinet_Ramis_82,Martinet_Ramis_83,Moussu_Cerveau_88}. 

From a technical point of view, the proof of our versal deformation theorem is a direct application of the general theory of 
normal forms developed in~\cite{Functors}. Therefore, with regards to the details of the proof of convergence, we refer to  that paper on the theory of Banach space valued functors.

\section{Divergent and convergent normal forms}
In the sixties, Bruno started the systematic investigation of differential 
equations over rings of formal power series, and complemented it with important
convergence resuls. But the formal case in itself is already a rich subject.
\subsection{The Poincar\'e-Dulac normal form}
We denote by $\widehat{\Ot}$ the local ring of formal 
power series in $d$-variables $x_1,x_2,\ldots,x_d$:
\[\widehat{\Ot}:=\CM[[x_1,x_2,\ldots,x_d]],\]
and by $\widehat{\Xt}$ the corresponding module of derivations
\[ \widehat{\Xt}:=Der(\widehat{\Ot})=\bigoplus_{i=1}^d \widehat{\Ot}\d_{x_i},\]
So a {\em formal vector field} $X \in \widehat{\Xt}$ is of the form
\[ X=\sum_{i=1}^d a_i(x) \d_{x_i},\;\;\;a_i(x) \in \widehat{\Ot},\] 
and can be written as an expansion of homogeneous components of increasing degree:
$$X=X_0+X_1+X_2+\dots $$
Here and in the sequel the dots stand for higher order terms in the Taylor 
expansion at the origin. 

The vector field is called {\em singular} if $X_0=0$. We say it is 
an {\em isolated singularity}, if the  origin is scheme-theoretically the only
point for which $X_0=0$, or equivalently if the ideal
\[ (a_1,a_2,\ldots,a_n) \subset \Oth\]
is $\Mt$-primary, where $\Mt$ denotes the maximal ideal of $\Oth$.
Such a vector field $X$ induces derivations $X_k$ of $\Mt/\Mt^{k+1}$ and as
any linear map, can be decomposed into semi-simple and nilpotent part:
$$X_k=S_{k}+N_{k},\;\;\; [S_k, N_k]=0,$$
and by taking the limit $k \to +\infty$, we obtain the {\em abstract Poincar\'e-Dulac normal form}
$$X=S+N,\;\;\; [S,N]=0. $$

\subsection{Resonances}
From now on we assume that the linear part of the vector field is semi-simple:
$$X_1=S_1=S,\ N_1=0 .$$ 
In appropriate coordinates the semi-simple part takes the form
$$S= \sum_{i=1}^d \l_i x_i \d_{x_i} ,$$
and we call 
$$\L:=(\l_1,\l_2,\ldots,\l_d)$$
the {\em frequency vector} of $X$.
We denote by
\[ {\Oth}^S:=\{f \in \Oth\;\;|\;\; S(f)=0\} \subset \Oth\]
the {\em ring of invariants} of $S$. A monomial
\[x^R:=x_1^{r_1} x_2^{r_r}\ldots x_d^{r_d},\;\;\;r_i \in \NM\]
is invariant if and only if
\[ (\L,R):=\sum_{i=1}^d \l_i r_i =0,\]
and the  ring $\Oth^S$ is generated by a finite number of such {\em resonant monomials}
$$x^{R_1},x^{R_2}, \ldots, x^{R_p},\; R_j \in \NM^d,$$
so that
\[\Oth^S=\CM[[x^{R_1}, x^{R_2},\ldots,x^{R_p}]] .\]
The vectors $R$ for which $(\L,R)=0$ are called {\em resonances} and all 
resonances are non-negative linear combinations of the basic resonances $R_i$, 
but this representation in general is not unique.

Similarly, we denote by
\[ \Xth^S:=\{ V \in \Xth\;\;|\;\;[S,V]=0\}\]
the {\em resonant vector fields}, which  are invariant under the infinitesimal
action by $S$. A monomial field 
\[ x^{K}\d_{x_i}= x_1^{k_1}x_2^{k_2}\ldots x_d^{k_d}\d_{x_i} \] 
belongs to $\Xth^S$ if and only if
$$(\Lambda,K)-k_i=0  .$$
Note that the fields $x_i\d_{x_i}, i=1,2,\ldots,n$, always belong to $\Xth^S$.
The resonant vector fields $\Xth^S$ form a module over $\Oth^S$, which
in fact is finitely generated. So there are finitely many resonant fields
\[V_1,V_2,\ldots,V_m,\] such that any $V\in \Xth^S$ can be expressed as
\[ V=\sum_{i=1}^m \phi_i(x^{R_1},X^{R_2},\ldots,x^{R_p}) V_i\]
 
 This immediately leads to the following result:
\begin{theorem}[\cite{Bruno_Poincare_Dulac,Dulac_1912}] Let $X$ be a formal vector field then there exists an automorphism $\p \in Aut(\Oth)$ and power series $\phi_1,\dots,\phi_m \in \CM[[u]]:=\CM[[u_1,\dots,u_p]]$
such that $\p(X)$ is a resonant vector field i.e.:
$$\p(X)=S+\sum_{i=1}^m \phi_i(x^{R_{1}},\dots,x^{R_{p}})V_i$$
\end{theorem}
In the statement of the theorem and in the sequel, the derivation $\p(X)$ is defined by the commutative diagram
$$\xymatrix{\Oth \ar[r]^X \ar[d]_-\p & \Oth \ar[d]^-\p \\
\Oth \ar[r]^{\p(X)} & \Oth 
}
$$
Note that in differential geometry, a derivation $X$ is interpreted as a vector field and the notation $\p_*X$ is used. As $\p(X)$ cannot be understood in another way, we adopt this simplified notation which is more adapted to our computations.
\subsection{Positivity conditions}
\label{SS::assumptions}
We make the following assumptions on the eigenvalue $\L$ of $S$:\\
P1) there are no algebraic relations between a minimal set of
generators for the ring of invariants. In other words, we assume that
\[ \Oth^S=\CM[[u_1,u_2,\ldots,u_p]]\]
is a power series ring.\\
P2) the module of $\Xth^S$ is freely generated by the $x_i\d_{x_i}$'s over the ring~$\Oth^S $.\\

The following easy reformulation show why one can think of P1) and P2) as positivity conditions.

\begin{lemma} The conditions P1) and P2) are respectively equivalent to the existence of 
integral vectors $R_1,\dots,R_p \in \NM^d$ with $(L,R_j)=0$, such that:\\
1) Each vector $J \in \NM^d$ with $(\L,J)=0$ can be uniquely written as
\[ J=\sum_{j=1}^p n_j R_j, \;\;\;n_j \in \NM .\]
2) Each vector $K \in \NM^d$ with $(\L,K)=k_i$ for some $i$, can be uniquely written as 
\[ K=\sum_{j=1}^d m_j R_j+E_j, \;\;\; m_j \in \NM,\]
where $E_1=(1,0,\dots,0),\dots,E_d=(0,\dots,0,1)$ denotes the standard basis vectors.
\end{lemma}
\subsection{Examples and counter-examples}
Let us consider a pair of opposite eigenvalues 
\[ S= \l x\d_x-\l y\d_y\]
The ring of invariants is generated by $xy$, the resonant fields are
generated by $x\d_x $ and $y\d_y$. Thus it satisfies the positivity conditions.
The formal normal form of a vector
field with this linear part is 
\[ (\l+g(xy))x\d_x-(\l+h(xy))y\d_y\]
This vector field is Hamiltonian for the symplectic structure $dx \wedge dy$ precisely when $g=h$. A direct generalisation arises when eigenvalues form opposite pairs
\[ S= \sum_{i=1}^d \l_i x_i\d_{x_i}-\l_i y_i\d_{y_i}\]
and are generic otherwise, in the sense that
$$\dim_\QM \sum_{i=1}^d \QM \l_i =d .$$
The ring of invariants is generated by $x_iy_i$, $i=1,2,\ldots,d$, the
resonant fields are the $x_i\d_{x_i}$, $y_i\d_{y_i}$, so that it is a case
where the the positivity conditions hold.
So the formal normal form of a vector field with this linear part is
\[ S=\sum_{i=1}^d(\l_i+g_i(x_1y_1,\ldots,x_dy_d))x_i\d_{x_i}-(\l_i+h_i(x_1y_2,\ldots,x_dy_d))y_i\d_{y_i}\]
Again the field is Hamiltonian for the symplectic structure
$$\omega=\sum_{i=1}^n dx_i \w dy_i $$
precisely when $g_i=h_i$.

A very different case is the $(1,1,-2)$-resonance:
\[S=x\d_x+y\d_y-2 z\d_z\]
The ring of invariants is generated by 
\[ u:=x^2z,\;\;v:=y^2z,\;\;w:=xyz.\]
Among these there is a relation
\[uv-w^2=0,\]
so the map 
$$ \CM^3 \to \CM^3,\ (x,y,z) \mapsto (x^2z,y^2z,xyz)$$ is not surjective, but lands on
the quadratic cone defined by the above equation. The first positivity condition is not satisfied. Also, the module 
of resonant vector fields is generated by
\[ x\d_x,\;\; y\d_y,\;\;z\d_z,\;\;x\d_y,\;\;y\d_x\]
The second positivity condition is not satisfied either and the relations
\[ v x\d_x = w y\d_x,\;\;\; u y\d_y= w x\d_y\]
show $\Xt^S$ is not a free module over the invariant ring $\Oth^S$, showing that
P2) does not hold either.

In \cite[Section 4]{Stolovitch_94}, Stolovitch seems to suggests that P1) implies P2).  
But if we consider the case $\L=(\a+\b, \a,\b)$, where $\a$ and $\b$ are 
$\QM$-independent and positive, there are no non-trivial invariants, $\Oth^S=\CM$, but there 
is a non-trivial resonant vector field, namely $yz\d_x$. So P1) is satisfied, but not P2). By adding extra variables one obtains examples with $\Oth^S \neq \CM$.

\subsection{The resonant cone}
If $S$ satisfies the positivity assumptions, then any vector field in Poincar\'e
normal form is of the form
\[X=S+N=\sum_{i=1}^d (\l_i+\phi_i(x^{R_1},x^{R_2},\ldots x^{R_p}))x_i\d_{x_i}\]
In particular, such vector fields are of the form 
\[ X=S+T,\;\;\;T=\sum_i^d a_i(x)x_i\d_{x_i},\;\;\; a_i(x) \in \Oth \]
that we will call {\em formal logarithmic vector fields}. Such a vector field
have a special property.

\begin{definition}
The {\em resonant ideal} is the ideal $I$ generated by the invariants of positive degree.
\[ I=(x^{R_1},x^{R_2},\ldots,x^{R_p}) \subset \Oth\]
The {\em resonant cone} is the formal scheme defined by the resonant ideal.
\end{definition}

\begin{proposition} A logarithmic vector field 
\[X=S+T,\;\;\;T=\sum_i^d a_i(x)x_i\d_{x_i}\] 
preserves the the resonant cone. 
\end{proposition}
\begin{proof}
$$(S+T)x^{R_j} =\left( (\L,R_j)+\sum_{i=1}^d a_i(x)r_{i,j}  \right) x^{R_j} \in (x^{R_1},\dots,x^{R_p}).$$
\end{proof}

\begin{corollary} If $S$ satisfies the positivity assumption, then any vector field 
\[X=S+\dots\] 
admits an invariant ideal isomorphic to the resonant ideal.
\end{corollary}
\begin{proof}
By positivity, the normal form of the vector field is logarithmic, so the result from the above
proposition. 
\end{proof}

\subsection{The Stolovitch normal form}
\label{SS::Stolovitch_nf}
Like for the Poincar\'e-Siegel theorems, the transformation bringing a vector field to Poincar\'e Dulac normal form is, as a general rule, divergent~\cite{Bruno,Ilyashenko_divergence,Perez_Marco_divergence}. Therefore it is important to have a less restrictive but {\em convergent} normal form. 

We denote
\[ \Ot:=\CM\{x_1,x_2,\ldots,x_d\}, \;\;\;\Xt:=Der(\Ot)=\oplus_{i=1}^d \Ot \d_{x_i}\]
the local ring of convergent power series and the module of analytic vector fields. 

The following fundamental result is due to Stolovitch~\cite{Stolovitch_94}:

\begin{theorem} Let $S$ be a  linear vector field that satisfies the Bruno and positivity conditions. 
Then for any analytic vector field
$$X =S+\dots \in \Der(\Ot)$$ 
there exists an analytic automorphism $\p \in \Aut(\Ot)$ such 
$$\p(X)=S\ \mod \Xt_{log} $$
where 
\[ \Xt_{log}=\sum_{i=1}^d \Ot x_i\d_{x_i} \]
is the $\Ot$-module of logarithmic vector fields.
\end{theorem}

As remarked above, the formal version of this theorem trivially follows from the Poincar\'e-Dulac normal form. The importance of the theorem lies in the following corollary.

\begin{corollary} If $S$ satisfies the positivity conditions, then any analytic vector 
field $X=S+\ldots$ admits an invariant analytic germ, isomorphic to the analytic 
resonant cone.
\end{corollary}

Here by the phrase {\em analytic resonant cone} we mean of course the germ defined by the ideal in $\Ot$ generated by the 
analytic resonant ideal $(x^{R_1},x^{R_2},\ldots,x^{R_p}) \subset \Ot$.\\

The above theorem of Stolovitch can be proven using our {\em Lie iteration}, whose convergence can be shown in great generality~(see also~\ref{SS::abstract}).  
 

\section{The formal versal deformation}
In this section, we consider a derivation with semi-simple linear part $S$. We assume positivity conditions on $S$ and choose an integer basis $R_1,\dots,R_j$ of resonances.
\subsection{Bruno variables}
In \cite{Bruno_resonance}, Bruno extended the ring $\Oth$ by adding variables $u_1,\dots,u_p$ corresponding to the $p$ monomial  invariants $x^{R_i},i=1,2,\ldots,p$.  In the approach of Stolovitch in \cite{Stolovitch_KAM}, this can be seen as a generalisation the variables introduced by Moser in KAM-theory to parametrise the invariant tori. 

Geometrically, we introduce a space $\CM^p$ with coordinates 
$u_1,\dots,u_p$ and define a map
\[  \CM^d \to \CM^p, x \mapsto (x^{R_1},x^{R_2},\ldots,x^{R_p}).\]
The (completion at $0$ of the) fibre over $0$ is the resonant cone; note that the semi-simple vector 
field $S$ preserves the fibres of $\varphi$.
The graph of this map is the space $\Sigma \subset \CM^{n+p}$ defined by the 
vanishing of the polynomials
\[ J_1:=u_1-x^{R_1},\;\;\;J_2:=u_2-x^{R_2},\ldots,\;\;\;J_p:=u_p-x^{R_p} .\]
Any vector field $Y$ on $\CM^d$ can be lifted to a vector field $\widetilde Y$
on $\CM^{d+p}$, tangent to $\Sigma$:
$$Y \mapsto \widetilde{Y}:=Y+\sum_{j=1}^p Y(x^{R_j})\d_{u_j} ,$$
reflected algebraically by the statement
\[ \widetilde{Y}(J_j)=0,\;\;j=1,2,\ldots,p .\]

In our formal context we extend the power series ring $\Oth$ by formal variables and work in
\[ \Oth[[u]]=\CM[[x,u]]=\Oth[[u_1,u_2,\ldots,u_p]]=\CM[[x_1,\ldots,x_d,u_1,\ldots,u_p]]\]
and rather than the manifold $\Sigma$, we consider the ideal
\[J:=(J_1,J_2,\ldots,J_p) \subset \Oth[[u]] .\]
Clearly, there is  a ring-isomorphism
$$\Oth[[u]]/J \to \Oth,\;\;\;x_i \mapsto x_i,\;\;\; u_j \mapsto x^{R_j}.$$
We will use the lift $Y \mapsto \widetilde Y $ to identify $\Xth$, $\Xth^S$ as subspaces of $\Der(\Oth[[u]])$. 
and suppress the\ $\widetilde{}$\ of the notation.

\begin{definition} 
The ideal $J \subset \Oth[[u]]$ is called the {\em graph ideal} of $S$.
\end{definition}
 
We assign degree $1$ to the variables $x_i$ and
$$|R_i|=\sum_{j=1}^d R_{i,j}$$ 
to the variable $u_j$, so that the polynomials $I_j$ are homogeneous of
degree $|R_j|$. We filter the ring $\Oth[[u]]$, by the {\em order} of a 
series, that is, the smallest degree of a monomial appearing in the series.

This filtration of the ring induces a filtration on $\Der(\Oth[[u]])$: 
a derivation is said to be of order $d$ if it maps terms of order $i$ to 
terms of order $i+d$, so that the order of $S$ is $0$.
\subsection{The detuning deformation}
We now follow Martinet's treatment of versal deformations and construct a ring with extra parameters to describe a versal deformation. If we are given an arbitrary deformation, given by a family with still
other variables, we also add these to the ring. The inducing maps are then obtained by solving 
equations by the implicit function theorem~\cite{Martinet}.  The method is classical and we used this 
technique also in~\cite{Hamiltonian_Normal_Forms}. 

Following Bruno, we work in the ring $\Oth[[u]]$ and with lifted fields $\Xth \subset Der(\Oth[[u]])$. 

By a {\em deformation} of $S$ we mean a vector field depending on $l$ parameters $\mu=(\mu_1,\mu_2,\ldots,\mu_l)$  of the form
$$X=S+T(\mu) \in \Xth[[\mu]],\;\;\;T(0)=0 .$$
Here 
\[ \Xth[[\mu]] \subset Der(\Oth[[u,\mu]])\]
consist of those vector fields of $\Xth[[u]]$ that contain no $\d_{\mu_i}$ (relative vector fields).

A special {\em detuning deformation} is obtained by introducing frequency variabes $\phi_1,\ldots,\phi_d$ and
considering 
\[S_v =S+\sum_{i=1}^d \phi_i x_i\d_{x_i}=\sum_{i=1}^d (\l_i + \phi_i) x_i\d_{x_i} \in \Xth[[\phi]] \]
Under our running positivity assumptions we have
\[ \Xth^S=\oplus_{i=1}^d \Oth^S x_i  \d_{x_i},\]
so the Poincar\'e-Dulac theorem suggests that the detuning deformation is versal in a formal sense. 

If we want to induce an arbitrary deformation $X=S+T(\mu)$ from the detuning deformation, one first forms
the {\em sum} of the two deformations, that is we consider
\[ X_v=S_v+T(\mu)=S+\sum_{i=1}^d\phi_i x_i\d_{x_i}+T(\mu),\]
a deformation of $S$ with two sets of parameters $\phi$ and $\mu$.
The versality is then expressed by the existence of a certain automorphism
\[ \varphi \in Aut(\Ot[[u,\phi,\mu]])\]
that maps $X_v$, considered as a deformation of $S_v$, back to $S_v$.
So we will have to work with variables $u,\phi$ and $\mu$ and vector fields that
depend on these variables, but there will be no derivatives in the $\mu$ or $\phi$ directions, and thus belong to
$$\Xth[[\phi,\mu]] \subset \Der_{\CM[[\phi,\mu]]}(\Oth[[u,\phi,\mu]]). $$
We also use the notation
$$ \Xth^S[[\phi,\mu]] \subset \Der_{\CM[[\phi,\mu]]}(\Oth[[u,\phi,\mu]])$$
for the $\Oth[[\phi,\mu,u]]$-module generated by the $(x^{R_i}-u_i)x_j\d_{x_j}$'s.
 
We assign degree $2$ to the $\mu$ variables, in this way making a deformation of a linear vector field is obtained by
adding terms of higher order. 
 \subsection{Formal normal form theorem}
 \label{SS::versal}
\begin{theorem}
\label{T::formal} Let $S=\sum \l_i x_i \d_{x_i}$ be a linear vector field satisfying the positivity assumptions
and let 
$$S_v=\sum_{i=1}^d (\l_i+\phi_i) x_i\d_{x_i} \in \Xth[[\phi]]   $$
be its detuning deformation. Then for any perturbation of the form
$$X_v:= S_v+\dots  \in \Xth[[\phi,\mu]],  $$
where the dots stand for higher order terms;
there exists an automorphism $\p$ of the algebra $\Oth[[\phi,\mu,u]]$ which has the following properties:
 
\begin{enumerate}[{\rm 1)}]
  \item $\p(X_v)= S_v \quad \mod\ \Xth^S[[\phi,\mu]] $.
  \item $\p(u_j)=u_j$ for $j=1,\dots,k$.
  \item $\p(\mu_j)=\mu_j$ for $j=1,\dots,l$.
  \item $\p(\phi_j)=\phi_j+\dots \in \CM[[\phi,\mu,u]].$
\end{enumerate}
\end{theorem}
\begin{proof}
The proof is done by applying the Lie iteration. To save ink we set:
\[ R:=\Oth[[u,\phi,\mu]] , \Der:=\Xth[[\phi,\mu]] \subset \Der(R) \]
The infinitesimal action on the versal deformation is given by map
$$\Der \to  \Der,\  X \mapsto [X,S_v].$$
It is $\CM[[u,\phi,\mu]]$-linear and diagonal in the monomial basis:
$$[x^K\d_{x_i},S_v]=\l_{K,i} x^K\d_{x_i},\;\;\;\l_{K,i}:=\left(\l_i+\phi_i-(\Lambda+\phi,K)\right) .$$
We define a $\CM[[u,\phi,\mu]]$-linear map $L$ by putting for non-resonant vector $K$ by
\begin{align*}
L(x^K\d_{x_i})&=\l^{-1}_{K,i}x^K\d_{x_i} \\
L( x_i\d_{x_i})&=\d_{\phi_i}
\end{align*}
As the module of resonant vector fields is freely generated by vector fields of the form $x_i\d_{x_i}$ the map $L$ is well-defined.
Note that
\begin{align*}
[u_j \d_{\phi_i},S_v]&=u_jx_i \d_{x_i} \\
&=x^{R_j}x_i\d_{x_i}+(u_j-x^{R_j})x_i\d_{x_i} \\
&=x^{R_j}x_i\d_{x_i} \quad \mod J
\end{align*}
where $J$ is the graph ideal.
Using $L$, for $V=S_v+T$, we define the map:
\[j_{V}: m \mapsto Lm-L(Lm(T))=L(m-Lm(T)).\]
This map satisfies
$$[j_V(Y),S_v+T]=Y\ \mod J $$
We define
$$X_0=X_v,\ S_0=0,\ A_0=S_v, \ v_0=[j_{S_v}(X_0)]_1^2 ,$$
where
\[ [-]_a^b\]
denotes the sum of terms of degree $\ge a$ and $<b$ in the Taylor series at $0$.
The following {\em Lie iteration} which brings a vector field $X$ back to its normal form~:
\vskip0.2cm
$$\begin{array}{ ll }
\ & \\
X_{n+1}&=e^{-[-,v_n]}X_n,\\
\ & \\
S_{n+1}&=[X_n-[A_n,v_n]]_{2^n}^{2^{n+1}},\\
\  & \\
A_{n+1}&=A_n+S_{n+1},\\
\ & \\
v_{n+1}&=j_{A_n}([X_n]_{2^{n+1}}^{2^{n+2}})\\
\ & \\
\end{array}$$
\end{proof}
\subsection{Relation to the Poincar\'e-Dulac normal form}
The theorem implies the classical Poincar\'e-Dulac theorem. In this case, our perturbation does not contain the $\mu$ parameters.

 We start a vector field 
 $$X=S+X',\;\;X \in \Xth$$
where $X'$ are terms of positive order and $S$ is semi-simple. We consider 
the associated perturbation of our corresponding versal deformation:
 $$X_v=S_v+X' $$
So  $X$ is obtained from $X_v$ by setting all $\phi_i=0$. The above
Theorem \ref{T::formal} tells us that there exists an automorphism $\varphi \in Aut(\Oth[[\phi,u]])$ such that
 $$\p(X_v)= S_v\;\; \mod \Xth^S[[\phi]]$$
This automorphism $\p \in Aut(R)$ maps $\phi_j$ to $\varphi(\phi_j)=\phi_j+\ldots \in \CM[[\phi,u]]$. By the implicit function theorem, we may solve  
the equations
$\p(\phi_j)=0$ by $\phi_j=g_j(u)$ for certain series $g_j(u)$. So we obtain
\begin{align*}
\p(X)&=\sum_{i=1}^d (\l_i+g_i(u))x_i \d_{x_i}+\sum_{j=1}^p (g,R_j)u_j\d_{u_j} \; \mod \Xth^S
\end{align*}
which is the usual Poincar\'e-Dulac normal form using Bruno variables.
\subsection{The versality property}
The result implies that the detuning deformation is formally versal. Indeed, let
$$X=S+T(\mu)\in \Der_\CM(\Oth[[\mu]]) $$
be an arbitrary deformation of $S$, so
$$X=\sum_{i=1}^d (\l_ix_i+T_i(\mu,x)) \d_{x_i},$$ 
such that $T_i(\mu=0)=0$. We  form the sum deformation:
$$X_v=S_v+T(\mu). $$
The theorem above states that $X_v$ is mapped via an automorphism $\p \in \Aut(\Oth[[u,\phi,\mu]])$ to $S_v$. Again the restriction of $X_v$ to $\phi=0$ is the initial deformation 
$X$. Solving the equations $\p(\phi_j)=0$ by $\phi_j=g_j(u,\mu)$, we get that
$$\p(\widetilde{X})= \sum_{i=1}^d (\l_i+g_i(u,\mu))x_i \d_{x_i}+(g,R_i)u_i\d_{u_i}\ \mod \Xt^S[[\mu]].$$
If we wish one may eliminate the $u$-variables using the ideal $J$ and get
\[ \p(X) = \sum_{i=1}^d (\l_i+g_i(x^R,\mu)) x_i\d_{x_i}\]
\subsection{The Bruno-Stolovitch ideal}
Consider a vector field in normal form
$$Y= \sum_{i=1}^d (\l_i+g_i(u,\mu))x_i \d_{x_i}+\sum_{j=1}^p(g,R_j)u_j\d_{u_j}\;\mod \Xt^S[[\mu]] $$
here $g=(g_1,\ldots,g_p)$, $R_j=(R_{j,1},\ldots, R_{j,p}$ and $(g,R_j)=\sum_{i=1}^p a_i R_{j,i} $. The graph ideal $J$ is by definition $Y$-invariant, but it contains also a bigger invariant ideal $I_{\mu}$ obtained by adding the functions $(g,R_i)$'s.
\begin{definition} The Bruno-Stolovitch ideal of the vector field $Y$ 
is the ideal
\[I_{\mu}=J+((g,R_1),\ldots,(g,R_p)) \subset \Oth[[u,\mu]]\]
\end{definition}

\begin{lemma} The ideal $I_{\mu}$ is $Y$-invariant:
$$f \in I_{\mu} \implies Y(f) \in I_{\mu} $$
\end{lemma} 
\begin{proof}
A direct computations shows that
$$Y(u_i)=(g,R_i) u_i \quad \mod J $$
and therefore $Y(u_i) \in I_{\mu}$.  In vector notations
$$Y(g)=(D_ug)Y(u) $$
which can be interpreted as the image of the vector field $Y$ under the differential $D_ug$ of the map $g$.
Therefore
$$Y((g,R_i))=((D_u g) Y(u),R_i) \quad \mod J $$
and the right hand side lies in $I_{\mu}$. This proves the lemma.
\end{proof}
This means that the ideal $I$ generated by the $x^{R_j} - u_j$'s and the 
functions $(g,R_j)$ is an invariant ideal of the vector field.

Assume for a moment that for a fixed value of $\mu$ the functions $g_i$'s are 
analytic in a small neighbourhood of the origin. Then the vector field $Y(\mu,-)$ 
is analytic and tangent to the {\em Bruno-Stolovitch varieties}:
$$V_\mu:=\{ (x,u) \in \CM^{d+p}: (g(u,\mu),R_i)=0,\ u=x^R \}. $$
As Bruno showed in~\cite{Bruno_resonance}, as a general rule, one cannot hope 
to have even a $\mu$-dependant analytic family of the varieties. However in 
the formal case, the varieties $V_\mu$ are formal schemes and, as one would 
expect from classical KAM theory, these form a Whitney $C^\infty$-family of 
analytic varieties over a positive measure set.
\subsection{Piartly's example}
The versal deformation
$$Y=(\a +\mu+xy+\dots ) x\d_{x}+(-\a +\mu+xy+\dots  )y\d_{y} .$$  
 is given by:
$$S_v=(\a +\phi_1 ) x\d_{x}+(-\a +\phi_2  )y\d_{y}   $$
and by versality of this deformation, there exists an automorphism $\p$ such that
$$\p(\widetilde{Y})=(\a +g_1(u,\mu) ) x \d_x+(-\a +g_2(u,\mu ))y\d_y+(g_1+g_2)u\d_u \ \mod M $$
As the only  resonant exponent is $(1,1)$, the Bruno-Stolovitch ideal is generated by 
$$((1,1),g)=g_1(u,\mu)+g_2(u,\mu)$$
and by the function $u-xy$ which generates the resonance ideal. 

An explicit computations shows that we have the expansion 
$$g_1(u,\mu)+g_2(u,\mu)=2u+\mu+\dots $$
The function 
$$g_1(xy,\mu)+g_2(xy,\mu)= 2xy+\mu+\dots$$
defining the Bruno ideal has an $A_1$-singularity.

If we assume this expansion to be analytic for some small value of of $\mu$ 
then the manifold $V_\mu$ is a cylinder for $\mu \neq 0$ and  a cone for 
$\mu=0$. Assuming $\a$ imaginary, we may 
restrict the field to the real plane $\RM^2 \approx \CM \subset \CM^2$ 
defined by the equation  
$$g_1(z\bar z,\mu)+g_2(z\bar z ,\mu)=2|z|^2+\mu+\dots$$
is diffeomorphic to a circle for $\mu<0$. We obtain the standard picture of
the family of vanishing cycles of the $A_1$-singularity (see~\cite{AVGII,AVGL2}).
 \subsection{The multi-dimensional Hopf bifurcation}
It is easy to generalise the above picture to higher dimensions. 
We start with a vector field of the form
$$Y= \sum_{i=1}^d(\a_i +\mu_i+x_iy_i+\dots ) x_i\d_{x_i}+(-\a_i +\mu_i+x_iy_i+\dots  )y_i\d_{y_i} $$ 
where the $\a_i$ are $\QM$-independent. 
We consider the versal deformation:
$$S_v= \sum_{i=1}^d(\a_i+\phi_i)x_i\d_{x_i}+\sum_{i=1}^d(-\a_i+\phi_{i+d})y_i\d_{y_i}.$$
The vector field $\widehat{Y}$ is  isomorphic to
$$Y= \sum_{i=1}^d(\a_i+g_i(u,\mu))x_i\d_{x_i}+\sum_{i=1}^d(-\a_i+g_{i+d}(u,\mu))y_i\d_{y_i}$$
The resonance ideal is generated by the $u_i-x_iy_i$'s and the Bruno-Stolovitch ideal is generated by these functions to which we add the functions $g_i+g_{i+d}$. In this case, if the function turns out to be analytic form some $\mu$-value and if we consider real structure $x_=\overline{y_i}$ as for $d=1$, the real parts of the manifolds $V_\mu$, for $\mu_i<0$, are tori 
and the vector field carries a quasi-periodic motion.

\section{Versal deformations and invariant varieties} 
Iteration above was defined on the level of formal power series. But one can define
a similar iteration defined for appropriate Banach spaces of analytic functions.
The convergence of the iteration procedure then is  a direct consequence of 
the general theory developed in~\cite{Functors} and that we will now briefly recall. 
We will be very sketchy as full details were already given in~\cite{Functors}.
It is likely that one may produce complete and direct proofs along the lines of 
Stolovitch~\cite{Stolovitch_94,Stolovitch_KAM}, but this would require a much longer
line of arguments and computations.
\subsection{Bruno sequences}
For convergence, our precedure requires Bruno-type conditions. To fix notations, we 
briefly recall the standard facts.
\begin{definition}[\cite{Bruno}]
A strictly monotone positive sequence $a$ is called a  {\em Bruno sequence} 
if the infinite product
\[ \prod_{k=0}^{\infty} a_k^{1/2^{k}}\] 
converges to a strictly positive number or equivalently if
$$\sum_{k\geq 0} \left| \frac{\log a_k}{2^{k}}\right| <+\infty. $$
\end{definition} 
We denote respectively by $\Bt^+$ and $\Bt^{-}$ the set 
of increasing and decreasing Bruno sequences.\\
Attached to a frequency vector $\L \in \CM^{d}$, we define the sequence
$\sigma(\L)$ with terms
$$\s(\L)_k :=\min_{i} \{ |(\L,J-E_i)| \neq 0: J \in \NM^{d} \setminus \{ 0 \}, \| J \| \leq 2^k \} $$
Here $E_i(=0,\dots,0,1,0,\dots,0)$ denotes the standard basis of the vector space $\CM^d$.
We say that the vector $\L$ is {\em Bruno} if the sequence $\s(\L)$ is a decreasing Bruno sequence.
\subsection{Functional analytic setting (part I)}
We fix the frequency vector $\L$ of our linear vector field $S$.
We will replace the the ring $\Oth[[u]]$ of formal power series by an 
appropriate system of Banach space of holomorphic functions. 
To describe these, we first consider the space $\CM^d  \times \CM^p$, with coordinates $x$, $u$. 
We put 
\[ D^d_s:=\{x \in \CM^d\;|\;|x_i|<s,\;i=1,2,\ldots,d \}\]
\[ D^p_s:=\{ u \in \CM^p\;|\;|u_j|<s,\;j=1,2,\ldots,p \}\]
We define 
$$U=D^d \times D^k \to \NM \times \RM_{>0}$$ 
as a relative set with fibres independent on $n \in \NM$:
$$U_{n,s}=D^d_s \times D^p_s $$
We then consider the Banach spaces of holomorphic functions and vector fields
on these sets (with continuous extension to the boundary),
\[ E_{n,t}:=\Ot^c(U_{n.t}),\;\;\;F_{n,t}:=\Xt^c(U_{n,t})\]
for which there are corresponding restriction maps of norm $\le 1$
\[ E_{n,t} \to E_{n,s},\;\; s < t,\;\;\;\;E_{m,s} \to E_{n,s},\;\; n > m\]
and similarly for $\Xt$. In the terminology of \cite{Functors} such a system of
Banach spaces is called an {\em Arnold space} and write simply $E=\Ot^c(W)$, etc.

The set $U_{n,s}$ has the {\em Huygens property} for the constant sequence $a_n=1$, by which we mean
\[U_{n,s}+ D_{a_n(t-s)}  \subset U_{n,t} .\]
This guarantees that the locality estimates are satisfied for differential
operators and when passing from one Arnold space to another.

In the iteration process, we will also have to choose an appropriate sequence
\[s_0 >s_1 >s_2 \ldots >s_n >s_{n+1}> \ldots > s_{\infty} >0\]
and form the 'pull-back', i.e. we form
\[ E_n:=E_{n,s_n}, n=0,1,2,\ldots\]
for which we still have restriction maps
\[ E_0 \to E_1 \to E_2 \to \ldots \]
As sequence $(x_n)$ with $x_n \in E_n$ is called convergent if it is a Cauchy 
sequence in the sense that
$$\forall \e>0,\ \exists N ,\; n\geq N \implies | \iota x_n -x_{n+p}| \leq \e $$
where $\iota=\iota_{n+p,n}$ is the restriction map from $E_{n}$ to $E_{n+p}$.
\subsection{The abstract versal deformation theorem}
\label{SS::abstract}
\begin{theorem}[\cite{Functors}] Let $\alg$ be an Arnold-Lie algebra, 
$X \in \alg$ and $\mathfrak{h} \subset \alg$ a sub-Arnold-Lie algebra which is a direct summand, that is,
the projection on $\mathfrak{h}$ is a local operator whose norm is an integrable Bruno
sequence. Assume that 
\begin{enumerate}[{\rm 1)}]
 \item there exists a  a local operator $L \in \Lt(\alg,{\bf M})$ and a cutoff operator
$\kappa \in \Kt(M,M)$
solving approximatively the homological equation modulo $\mathfrak{h}$, in the sense that:
$$[L(Y),X]=Y+\kappa(Y) \quad \mod \mathfrak{h}$$
\item the norm sequence $|L|,\ |\kappa|$ are integrable Bruno sequences.
\end{enumerate}
Then there exists $r \in \RM_{>0}$ and $k$ such that for any $Y \in rt^k(B_{\alg})_t$ there exists $v \in \G(\NM,\alg)$ satisfying
$$e^v(X+rY )=\iota X\ \mod \mathfrak{h}$$
\end{theorem}
We will not give all the details here but we will indicate which functional space should be considered to deduce the theorem as in the next paragraph we prove a stronger result with parameters for families of vector fields.


Appyling the abstract versal deformation theorem with 
$\alg=\Xt^c(U)$ and $\mathfrak{h}$ the $\Ot^c(U)$-module generated by $(\Xt^S)^c(U)$, we deduce the Stolovitch normal form theorem \ref{SS::Stolovitch_nf}.
 
\subsection{Functional analytic setting (part II)}
We will now formulate and sketch the proof of the above theorem with additional
parameters $\phi$ and $\mu$. We have to make only notational changes in our set-up-

For the domain of $\phi$-variables, we have to throw away the appropriate 
neighbourhoods of the resonant hyperplanes. We fix a decreasing sequence 
$a=(a_n)$ and $s_0>0$ and consider the set 
$$Z \to \NM \times ]0,s_0] $$
with fibres
\[Z_{n,s}:=Z_{n,s}(\Lambda,a,s_0):=\{\phi \in D_s^{d}: \forall k \leq n,\s(\Lambda+\phi)_k  \geq  a_k(s_0-s) \} \]
This set has the {\em Huygens property}, by which we mean
\[Z_{n,s}+ D_{a^*_n(t-s)}  \subset Z_{n,t} ,\;\;\; a^*_n:=\frac{a_n}{2^n}.\]
The addition of the $\mu$-variables is straightforward. We consider
 \[ D^l_s:=\{ \mu \in \CM^l\;|\;|\mu_j|<s,\;j=1,2,\ldots,l \}\]
and get a relative set
 $$W \to \NM \times ]0,s_0] $$
 with fibre
\[ W_{n,s}:=D^d_s \times Z_{n,s} \times D^p_s  \times D^l_s \subset \CM^d \times \CM^d \times \CM^p\times \CM^l,\]
where $\iota=\iota_{n+p,n}$ is the restriction map from $E_{n}$ to $E_{n+p}$.
 \subsection{Versal deformations over a Cantor set}

Applying the abstract versal deformation theorem with 
 $$\alg=\Xt^c(W), \mathfrak{h}=(\Xt^S)^c(W)$$  and the map $L$ defined analoguous to
the $L$ in \ref{SS::versal}, we get the following result:
 
\begin{theorem}
\label{T::convergent} Let $\Lambda \in \CM^d$ be a Bruno vector satisfying the positivity conditions.
Consider a derivation
$$S_v=\sum_{i=1}^d (\l_i+\phi_i) x_i\d_{x_i} \in \Xt^c(W)_0   $$
of the algebra $\Ot^c(W)$. Then
for any other derivation of the form
$$X_v:= S_v+\dots \in \Xt^c(W)_0  $$
 there exists a morphism 
 $$\p:\Xt^c(W)_0 \to \Xt^c(W)_\infty$$  which has the following properties:
 \begin{enumerate}[{\rm 1)}]
  \item $\p(  X_v)= S_v \quad \mod\ (\Xt^c(W))^S $.
  \item $\p(u_j)=u_j$ for $j=1,\dots,k$.
  \item  $\p(\mu_j)=\mu_j$ for $j=1,\dots,l$.
  \item $\p(\phi_j)=\phi_j+\dots \in \Ot^c(D^{d+p})$
 \end{enumerate}
 \end{theorem} 
 As a corollary we get that the functions defining the Bruno-Stolovitch ideal lie in $\Ot^c(W)_\infty$, in particular they define a continuous family of invariant manifolds. Assuming non-degeneracy conditions on the frequency (for instance that the map is a submersion), these are parametrised  by a positive measure set. For the particular case of the multi-dimensional Hopf bifurcation, we get a family of tori carrying quasi-periodic motions generalising the theorem of  Chenciner-Li~\cite{Chenciner_Hopf,Li_Hopf}.  Stolovitch KAM theorem~\cite{Stolovitch_KAM} can also be deduced from the convergence of the Lie iteration on this setting.
\bibliographystyle{plain}
\bibliography{master.bib}
\end{document}